\documentclass[a4paper,12pt]{article}

\usepackage[left=2cm,right=2cm, top=2cm,bottom=3cm,bindingoffset=0cm]{geometry}

\usepackage{verbatim}
\usepackage{amsmath}
\usepackage{amsthm}
\usepackage{amssymb}
\usepackage{delarray}
\usepackage{cite}
\usepackage{hyperref}
\usepackage{mathrsfs}
\usepackage{tikz}
\usetikzlibrary{patterns}
\usepackage{caption}
\DeclareCaptionLabelSeparator{dot}{. }
\newcommand{\al}{\alpha}

\newcommand{\ga}{\gamma}
\newcommand{\de}{\delta}
\newcommand{\la}{\lambda}

\newcommand{\eps}{\varepsilon}
\newcommand{\vv}{\varphi}
\newcommand{\iy}{\infty}

\theoremstyle{plain}

\numberwithin{equation}{section}

\newtheorem{thm}{Theorem}[section]
\newtheorem{lem}[thm]{Lemma}

\newtheorem{cor}[thm]{Corollary}

\theoremstyle{definition}

\newtheorem{example}[thm]{Example}

\newtheorem{ip}[thm]{Inverse Problem}

\theoremstyle{remark}

\newtheorem{remark}[thm]{Remark}

\DeclareMathOperator*{\Res}{Res}

 \allowdisplaybreaks

\begin{document}

\begin{center}
{\large\bf Local solvability and stability of the inverse problem for the non-self-adjoint Sturm-Liouville operator}
\\[0.2cm]
{\bf Natalia P. Bondarenko} \\[0.2cm]
\end{center}

\vspace{0.5cm}

{\bf Abstract.} We consider the non-self-adjoint Sturm-Liouville operator on a finite interval. The inverse spectral problem is studied, which consists in recovering this operator from its eigenvalues and generalized weight numbers. We prove local solvability and stability of this inverse problem, relying on the method of spectral mappings. Possible splitting of multiple eigenvalues is taken into account.
  
\medskip


\medskip


\vspace{1cm}

\section{Introduction}

The paper concerns the theory of inverse spectral problems for differential operators. Such problems consist in constructing operators by their spectral information.

We consider the Sturm-Liouville boundary value problem (BVP) $L = L(q(x), h, H)$ in the following form
\begin{gather} \label{eqv}
    -y'' + q(x) y = \la y, \quad x \in (0, \pi), \\ \label{bc}
    y'(0) - h y(0) = 0, \quad y'(\pi) + H y(\pi) = 0.
\end{gather}
Here $\la$ is the spectral parameter, $q$ is the complex-valued function from $L_2(0, \pi)$, called {\it the potential}, $h$ and $H$ are complex numbers. Denote by $\{ \la_n \}_{n = 0}^{\iy}$ the eigenvalues of $L$ counted with their multiplicities and numbered so that $|\la_n| \le |\la_{n+1}|$, $n \ge 0$.
The eigenvalue problem \eqref{eqv}-\eqref{bc} appears after separation of variables in problems of mathematical physics, describing wave propagation, heating and other processes.

The most complete results on inverse problems of spectral analysis have been obtained for {\it self-adjoint} Sturm-Liouville operators. Those results include uniqueness theorems, algorithms for constructive solution, spectral data characterization, local solvability and stability (see the monographs \cite{Mar77, Lev84, PT87, FY01} and references therein).

Let us formulate one of the classical inverse problems for $L$ in the self-adjoint case (i.e. when $q(x)$ for $x \in (0, \pi)$, $h$ and $H$ are real). In this case, all the eigenvalues $\{ \la_n \}_{n = 0}^{\iy}$ are real and simple. Let $\vv(x, \la)$ be the solution of equation~\eqref{eqv}, satisfying the initial conditions $\vv(0, \la) = 1$, $\vv'(0, \la) = h$.
Define the weight numbers
\begin{equation} \label{defal}
\al_n := \int_0^{\pi} \vv^2(x, \la_n) \, dx, \quad n \ge 0.
\end{equation}

\begin{ip} \label{ip:sd}
Given the data $\{ \la_n, \al_n \}_{n = 0}^{\iy}$, construct $q(x)$, $h$ and $H$.
\end{ip}

Inverse problem~\ref{ip:sd} is equivalent to the inverse problem by the spectral function, studied by Marchenko \cite{Mar77}.

In the {\it non-self-adjoint case}, some of the eigenvalues $\{ \la_n \}_{n = 0}^{\iy}$ can be multiple, so the problem becomes more difficult for investigation. There are significantly less studies on inverse problems for the non-self-adjoint operator \eqref{eqv}-\eqref{bc}. In particular, in \cite{FY01} the classical results for the self-adjoint Sturm-Liouville operator are generalized to the non-self-adjoint case with simple eigenvalues. Tkachenko \cite{Tk02} developed the method for solving another inverse problem, which consists in recovering the non-self-adjoint Sturm-Liouville operator from two spectra corresponding to different boundary conditions.

Note that the spectral data $\{ \la_n, \al_n \}_{n = 0}^{\iy}$ defined above do not uniquely specify $q$, $h$ and $H$ in the general case.
Nevertheless, in \cite{But07} the so-called {\it generalized spectral data} (GSD) has been introduced in the following way.
Without loss of generality, we assume that multiple eigenvalues are consecutive: $\la_n = \la_{n + 1} = \ldots = \la_{n + m_n - 1}$, where $m_n$ is the multiplicity of the eigenvalue $\la_n$. By virtue of the well-known asymptotics
\begin{equation} \label{asymptla}
\sqrt{\la_n} = n + O\left( n^{-1}\right), \quad n \to \iy,
\end{equation}
we have $m_n = 1$ for sufficiently large $n$.
Define
\begin{gather*}
S := \{ 0 \} \cup \{ n \in \mathbb N \colon \la_n \ne \la_{n-1} \}, \\
\vv_{n + \nu}(x) := \frac{1}{\nu!} \frac{d^{\nu}}{d\la^{\nu}} \vv(x, \la)_{|\la = \la_n}, \quad n \in S, \quad \nu = \overline{0, m_n - 1}.
\end{gather*}

The sequence $\{ \vv_n \}_{n = 0}^{\iy}$ is a complete system of root-functions for the problem $L$. The {\it generalized weight numbers} are defined as follows
\begin{equation} \label{defgal}
\al_{n + \nu} = \int_0^{\pi} \vv_{n + \nu}(x) \vv_{n + m_n - 1}(x) \, dx, \quad n \in S, \quad \nu = \overline{0, m_n-1}.
\end{equation}
Clearly, the definition~\eqref{defgal} generalizes \eqref{defal}.

Thus, Inverse Problem~\ref{ip:sd} turns into the inverse problem by GSD $\{ \la_n, \al_n \}_{n = 0}^{\iy}$. Buterin \cite{But07} has proved the uniqueness theorem for this inverse problem and obtained a constructive algorithm for its solution, based on the method of spectral mappings \cite{FY01, Yur02}. The question of GSD characterization for Sturm-Liouville operators with complex-valued potentials was investigated in \cite{BSY13, AHM08}. However, necessary and sufficient conditions on GSD from \cite{BSY13, AHM08} require solvability of some main equations. Those requirements are difficult to verify.

The aim of this paper is to investigate local solvability and stability of Inverse Problem~\ref{ip:sd} in the non-self-adjoint case. Note that, under a small perturbation of the spectrum, multiple eigenvalues can split into smaller groups, so the generalized weight numbers change their form. As far as we know, this effect was not studied before. 

Some fragmentary results on stability under splitting of multiple eigenvalues were obtained in \cite{MW05, HK11, BB17} for various inverse problems. Recently Buterin and Kuznetsova \cite{BK19} proved local solvability and stability for the inverse problem by two spectra for the non-self-adjoint Sturm-Liouville operator. They also took splitting of multiple eigenvalues into account. 
However, Inverse Problem~\ref{ip:sd} appears to be more interesting for investigation because of generalized weight numbers changing their structure.

In \cite{BSY13}, some results on local solvability and stability were obtained for the inverse problem of recovering the non-self-adjoint Sturm-Liouville operator with the Dirichlet boundary conditions from GSD. However, the authors of~\cite{BSY13} considered only such perturbations of GSD that preserve eigenvalue multiplicities. In the present paper, arbitrary perturbations are studied, that can change eigenvalue multiplicities. We obtain special conditions on a GSD perturbation, which allow GSD to change their structure, but the perturbation of the potential remains small in $L_2$-norm.

The paper is organized as follows. In Section~2, our main Theorems~\ref{thm:loc} and~\ref{thm:disc} on local solvability and stability are formulated. Their proofs are constructive and develop the ideas of the method of spectral mappings \cite{FY01, Yur02, But07}. This method consists in reduction of a nonlinear inverse problem to a linear equation in a Banach space, called {\it the main equation}. The main equation and the corresponding Banach space are specially constructed for every certain inverse problem. For our problem, the most important feature of the main equation is that it contains ``continuous'' and ``discrete'' components. Derivation of the main equation and the proof of its unique solvability are provided in Section~3. In Section~4, we finish the proofs of Theorems~\ref{thm:loc} and~\ref{thm:disc}, and consider an example. In Section~5, analogous results are formulated for equation~\eqref{eqv} with the Dirichlet boundary conditions.

\section{Main Results}

We start with some preliminaries. Let $\Phi(x, \la)$ be the solution of Eq.~\eqref{eqv}, satisfying the boundary conditions $\Phi'(0, \la) + h \Phi(0, \la) = 1$, $\Phi'(\pi, \la) + H \Phi(\pi, \la) = 0$. The function $M(\la) := \Phi(0, \la)$ is called the {\it Weyl function} of the problem $L$. Weyl functions are natural spectral characteristics for various self-adjoint and non-self-adjoint operators (see \cite{Mar77, FY01}). In \cite{But07}, the following representation has been obtained:
$$
M(\la) = \sum_{n \in S} \sum_{\nu = 0}^{m_n - 1} \frac{M_{n + \nu}}{(\la - \la_n)^{\nu + 1}}.
$$
The coefficients $\{ M_n \}_{n = 0}^{\iy}$ can be uniquely determined by the generalized weight numbers $\{ \al_n \}_{n = 0}^{\iy}$ and vice versa from the linear system
$$
\sum_{k = 0}^{\nu} \al_{n + \nu - k} M_{n + m_n - k - 1} = \de_{\nu, 0}, \quad n \in S, \quad \nu = \overline{0, m_n - 1}.
$$
In particular, $M_n = \al_n^{-1}$, if $m_n = 1$.
Thus, Inverse Problem~\ref{ip:sd} by the GSD is equivalent to the following one.

\begin{ip} \label{ip:M}
Given the data $G := \{ \la_n, M_n \}_{n = 0}^{\iy}$, find $q$, $h$ and $H$.
\end{ip}

Further we study Inverse Problem~\ref{ip:M} instead of Inverse Problem~\ref{ip:sd}.

Along with the problem $L$, we consider complex numbers $\tilde G = \{ \tilde \la_n, \tilde M_n \}_{n = 0}^{\iy}$. We will show that, if the data $\tilde G$ are ``sufficiently close'' to $G = \{ \la_n, M_n \}_{n = 0}^{\iy}$ in some sense (a rigorous formulation is given in Theorem~\ref{thm:loc} below), then $\tilde G$ will correspond to some BVP $\tilde L = L(\tilde q(x), \tilde h, \tilde H)$ of the same form as $L$, but with different coefficients. We agree that, if a certain symbol $\ga$ is related to $L$, then the symbol $\tilde \ga$ with tilde is the analogous object constructed by the data $\tilde G$. 

Let $N = N(L) \ge 0$ be the minimal integer such that $m_n = 1$ for $n > N$ and $|\la_N| < |\la_{N+1}|$. Fix any real $r \in (|\la_N|, |\la_{N + 1}|)$, and consider the contour $\ga_N = \ga_N(L) = \{ \la \in \mathbb C \colon |\la| = r \}$ surrounding the eigenvalues $\{ \la_n \}_{n = 0}^N$. Put 
\begin{gather} \nonumber
S_N := S \cap \{ 0, \dots, N \}, \quad \tilde S_N := \tilde S \cap \{ 0, \dots, N \}, \\ \label{defMN}
{\mathcal M}_N(\la) := \sum_{n \in S_N} \sum_{\nu = 0}^{m_n - 1} \frac{M_{n + \nu}}{(\la - \la_n)^{\nu + 1}}, \quad
\tilde {\mathcal M}_N(\la) := \sum_{n \in \tilde S_N} \sum_{\nu = 0}^{\tilde m_n - 1} \frac{\tilde M_{n + \nu}}{(\la - \tilde \la_n)^{\nu + 1}}, \quad
\hat {\mathcal M}_N := \tilde {\mathcal M}_N - {\mathcal M}_N.
\end{gather}
Note that the function $\hat {\mathcal M}_N(\la)$ is constructed by the data $\{ \la_n, M_n \}_{n  = 0}^N$ and $\{ \tilde \la_n, \tilde M_n \}_{n = 0}^N$.
Define $\rho_n := \sqrt{\la_n}$, $\arg \rho_n \in \left[-\tfrac{\pi}{2}, \tfrac{\pi}{2}\right)$, and
$\xi_n := |\rho_n - \tilde \rho_n| + |M_n - \tilde M_n|$ for $n \ge 0$.

\begin{thm} \label{thm:loc}
Let $L = L(q(x), h, H)$ be a fixed BVP in the form~\eqref{eqv}-\eqref{bc}, $N = N(L)$, $\ga_N = \ga_N(L)$. Then there exists $\de_0 > 0$ (depending on $L$) such that for any $\de \in (0, \de_0]$ and any complex numbers $\tilde G = \{ \tilde \la_n, \tilde M_n \}_{n = 0}^{\iy}$, satisfying the conditions
\begin{gather} \label{estM}
\max_{\la \in \ga_N} |\hat {\mathcal M}_N(\la)| \le \de, \\ \label{estxi}
\left( \sum_{n = N + 1}^{\iy} (n \xi_n)^2 \right)^{1/2} \le \de,
\end{gather}
there exist a complex-valued function $\tilde q \in L_2(0, \pi)$ and complex numbers $\tilde h, \tilde H$ being the solution of Inverse Problem~\ref{ip:M} for $\tilde G$. Moreover,
\begin{equation} \label{stab}
\| q - \tilde q \|_{L_2(0, \pi)} \le C \de, \quad |h - \tilde h| \le C \de, \quad |H - \tilde H| \le C \de.
\end{equation}
\end{thm}

Here and below the same symbol $C$ is used for various positive constants depending on $L$ and $\de_0$ and independent of $\de$, $\tilde G$, etc.

Recall that the function ${\mathcal M}_N(\la)$ is fixed, and all its poles lie inside $\ga_N$. The condition~\eqref{estM} for sufficiently small $\de$ implies that all the poles of $\tilde {\mathcal M}_N(\la)$ also lie inside $\ga_N$. Moreover, the following estimate holds:
$$
|\la_n - \tilde \la_n| \le C \de^{1/(N+1)}, \quad n = \overline{0, N}.
$$
However, the values $\la_n$ and $\tilde \la_n$ can have different multiplicities. Namely, multiple values $\la_n$ can split into smaller groups, so $S \subseteq \tilde S$, $\tilde m_n = m_n = 1$ for all $n > N$.

We also obtain local solvability and stability conditions on the discrete data, not involving the continuous function $\hat {\mathcal M}_N$. Such conditions are provided in the following theorem.

\begin{thm} \label{thm:disc}
Let $L = L(q(x), h, H)$ be a fixed BVP in the form~\eqref{eqv}-\eqref{bc}, $N = N(L)$. Then there exists $\de_0 > 0$ (depending on $L$) such that for any $\de \in (0, \de_0]$ and any complex numbers $\tilde G = \{ \tilde \la_n, \tilde M_n \}_{n = 0}^{\iy}$, satisfying the conditions \eqref{estxi} and
\begin{equation} \label{dist} 
\tilde \la_n \ne \tilde \la_k, \quad n \ne k, \quad n, k \ge 0, 
\end{equation}
\begin{equation} \label{estMla}
\def\arraystretch{1.5}
\left.
\begin{array}{l}
\biggl| \sum\limits_{j = 0}^{m_k - 1} (\tilde \la_{k + j} - \la_k)^s \tilde M_{k + j} - M_{k + s}\biggr| \le \de, \quad s = \overline{0, m_k - 1}, \\
\biggl| \sum\limits_{j = 0}^{m_k - 1} (\tilde \la_{k + j} - \la_k)^s \tilde M_{k + j}\biggr| \le \de, \quad s = \overline{m_k, 2(m_k - 1)}, \\
|\tilde \la_{k + j} - \la_k| \le \de^{1/m_k}, \quad |\tilde M_{k + j}| \le \de^{(1 - m_k)/m_k},
\end{array}
\right\} \quad k \in S_N,
\end{equation}
there exist a complex-valued function $q \in L_2(0, \pi)$ and complex numbers $h, H$ being the solution of Inverse Problem~\ref{ip:M} for $\tilde G$. Moreover, the estimates~\eqref{stab} hold.
\end{thm}

The condition~\eqref{dist} is imposed for simplicity. One can similarly consider the case of multiple values among $\{ \tilde \la_n \}_{n = 0}^N$, but then one needs more complicated requirements instead of the relations~\eqref{estMla}.

Theorems~\ref{thm:loc} and~\ref{thm:disc} generalize their analogue for the self-adjoint case \cite[Theorem~1.6.4]{FY01}.

\section{Main Equation}

The goal of this section is to derive the main equation in a Banach space, which plays a crucial role in the proofs of the main results. Our approach is based on the method of spectral mappings (see \cite{Yur02}).
Since a part of the proofs repeat the standard technique of \cite[Section 1.6]{FY01} and \cite{But07}, we omit the details and focus on the differences of our methods from the classical ones. 

Let us consider two BVPs $L = L(q(x), h, H)$ and $\tilde L = (\tilde q(x), \tilde h, \tilde H)$ with different coefficients. Fix $N = N(L)$ and the contour $\ga_N$. Assume that the eigenvalues $\{ \tilde \la_n \}_{n = 0}^{\iy}$ lie inside $\ga_N$, and $\{ \tilde \la_n \}_{n = N+1}^{\iy}$ lie outside $\ga_N$.

Define 
$$
D(x, \la, \xi) := \frac{\vv(x, \la) \vv'(x, \xi) - \vv'(x, \la) \vv(x, \xi)}{\la - \xi} = 
\int_0^x \vv(t, \la) \vv(t, \xi) \, dt.
$$

Consider the contour $\ga := \{ \la \in \mathbb C \colon \mbox{Re} \, \la = -p, \, \mbox{Im} \, \la \in [-p, p] \} \cup \{ \la \in \mathbb C \colon \mbox{Re} \, \la \ge -p, \, \mbox{Im}\la = \pm p \}$ with the counter-clockwise circuit. The constant $p$ is chosen so that all the eigenvalues $\{ \la_n \}_{n = 0}^{\iy}$ and $\{ \tilde \la_n \}_{n = 0}^{\iy}$ lie inside $\ga$.
Using the contour integration, we obtain the relation
$$
\vv(x, \la) = \tilde \vv(x, \la) + \frac{1}{2 \pi i} \int_{\ga} D(x, \la, \xi) \hat M(\xi) \tilde \vv(x, \xi) \, d\xi,
$$
where $\hat M := \tilde M - M$. Applying the Residue Theorem and noting that the function $(\hat M(\la) - \hat {\mathcal M}_N(\la))$ is analytic inside $\ga_N$, we obtain the relation
\begin{align} \nonumber
    \vv(x, \la) & = \tilde \vv(x, \la) + \frac{1}{2 \pi i} \int_{\ga_N} D(x, \la, \xi) \hat {\mathcal M}_N(\xi) \tilde \vv(x, \xi) \, d\xi \\ \label{cont} & + \sum_{n = N + 1}^{\iy} (\tilde M_n D(x, \la, \tilde \la_n)  \tilde \vv(x, \tilde \la_n) - M_n D(x, \la, \la_n) \tilde \vv(x, \la_n)).
\end{align}

We use the relation~\eqref{cont} for deriving the main equation of Inverse Problem~\ref{ip:M} in a special Banach space. Denote by $B_C$ the Banach space of functions continuous on $\ga_N$ with the norm
$$
\| f_C \|_{B_C} = \max_{\la \in \ga_N} |f_C(\la)|, \quad f_C \in B_C.
$$
Denote by $B_D$ the Banach space of bounded infinite sequences $f_D = [ f_n ]_{n = 1}^{\iy}$ with the norm
$$
\| f_D \|_{B_D} = \sup_{n \ge 1} |f_n|, \quad f_D \in B_D.
$$
Define the Banach space
$$
B := \{ f = (f_C, f_D) \colon f_C \in B_C, f_D \in B_D \},
\quad \| f \|_B := \| f_C \|_C + \| f_D \|_D.
$$
Here and below the lower indices $C$ and $D$ mean a ``continuous'' and a ``discrete'' part, respectively.

For every $x \in [0, \pi]$, define the element $\psi(x) = (\psi_C(x), \psi_D(x))$, where
\begin{gather*}
\psi_C(x, \la) = \vv(x, \la), \quad \la \in \ga_N, \qquad
\psi_D(x) = [\psi_n(x)]_{n = 1}^{\iy}, \\
\psi_{2j-1}(x) = \vv(x, \tilde \la_{N + j}), \quad
\psi_{2j}(x) = \chi_{N + j} (\vv(x, \la_{N + j}) - \vv(x, \tilde \la_{N + j})), \quad j \ge 1, \\
\chi_n := \begin{cases}
            \xi_n^{-1}, \quad \text{if} \: \xi_n \ne 0, \\
            0, \qquad \text{if} \: \xi_n = 0.
          \end{cases}
\end{gather*}

The element $\tilde \psi(x)$ is defined analogously, by using $\tilde \vv$ instead of $\vv$.

For the solution $\vv(x, \la)$, the following standard asymptotics is valid:
\begin{equation} \label{asymptphi}
\vv(x, \la) = \cos \rho x + O\left( \rho^{-1} \exp(|\mbox{Im}\, \rho| x)\right), \quad |\rho| \to \iy,
\end{equation}
where $\rho = \sqrt{\la}$, $\mbox{Re}\,\rho \ge 0$. Using~\eqref{asymptla} and \eqref{asymptphi}, we obtain the estimates
$$
|\vv(x, \la_n)| \le C, \quad |\vv(x, \la_n) - \vv(x, \tilde \la_n)| \le C \xi_n, \quad x \in [0, \pi], \quad n \ge 0, 
$$
where the constant $C$ does not depend on $x$ and $n$. Analogous relations are valid for $\tilde \vv(x, \la)$.
Consequently, for each fixed $x \in [0, \pi]$, we have $\psi(x) \in B$ and $\tilde \psi(x) \in B$. 

For each fixed $x \in [0, \pi]$, we define the linear bounded operator $R(x) \colon B \to B$ as follows:
\begin{gather*}
    R(x) = \begin{pmatrix}
                R_{CC}(x) & R_{CD}(x) \\
                R_{DC}(x) & R_{DD}(x)
           \end{pmatrix},
    \\ R_{CC}(x) \colon B_C \to B_C, \quad R_{CD}(x) \colon B_D \to B_C, \quad 
    R_{DC}(x) \colon B_C \to B_D, \quad R_{DD}(x) \colon B_D \to B_D, \\
    R(x) f = (R_{CC}(x) f_C + R_{CD}(x) f_D, R_{DC}(x) f_C + R_{DD}(x) f_D), \quad
    f = (f_C, f_D) \in B, 
\end{gather*}
\begin{align} \label{RCC}
    (R_{CC}(x) f_C)(\la) & = \frac{1}{2 \pi i} \int_{\ga_N} D(x, \la, \xi) \hat {\mathcal M}_N(\xi) f_C(\xi) \, d\xi, \\ \label{RCD}
    (R_{CD}(x) f_D)(\la) & = \sum_{k = 1}^{\iy} ((\tilde M_{N + k} D(x, \la, \tilde \la_{N + k}) - M_{N + k} D(x, \la, \la_{N + k})) f_{2k-1} \\ \nonumber & - \xi_{N + k} M_{N + k}D(x, \la, \la_{N + k}) f_{2k}), \\ \nonumber
    (R_{DC}(x) f_C)_{2j-1} & = \frac{1}{2 \pi i} \int_{\ga_N} D(x, \tilde \la_{N + j}, \xi) \hat {\mathcal M}_N(\xi) f_C(\xi) \, d\xi, \\ \nonumber
    (R_{DC}(x) f_C)_{2j} & = \frac{1}{2\pi i} \int_{\ga_N} (D(x, \la_{N + j}, \xi) - D(x, \tilde \la_{N + j}, \xi)) \chi_{N + j} \hat {\mathcal M}_N(\xi) f_C(\xi) \, d\xi, \\ \nonumber
    (R_{DD}(x) f_D)_{2j-1} & = \sum_{k = 1}^{\iy} (( \tilde M_{N + k}D(x, \tilde \la_{N + j}, \tilde \la_{N + k}) - M_{N + k} D(x, \tilde \la_{N + j}, \la_{N + k})) f_{2k-1} \\ \nonumber
    & - \xi_{N + k} M_{N + k} D(x, \tilde \la_{N + j}, \la_{N + k}) f_{2k}), \\ \nonumber
    (R_{DD}(x) f_D)_{2j} & = \chi_{N + j} \sum_{k = 1}^{\iy} ((\tilde M_{N + k} (D(x, \la_{N + j}, \tilde \la_{N + k}) - D(x, \tilde \la_{N + j}, \tilde \la_{N + k})) \\ \nonumber & - M_{N + k} (D(x, \la_{N + j}, \la_{N + k}) - D(x, \tilde \la_{N + j}, \la_{N + k}))) f_{2k-1} \\ \nonumber & - \xi_{N + k} M_{N + k} (D(x, \la_{N + j}, \la_{N + k}) - D(x, \tilde \la_{N + j}, \la_{N + k})) f_{2k}),
\end{align}
where $\la \in \ga_N$, $j \ge 1$, $f_D = [ f_k ]_{k = 1}^{\iy}$.

Taking $\la \in \ga_N$, $\la = \tilde \la_n$ and $\la = \la_n$, $n > N$, in \eqref{cont}, we obtain the so-called {\it main equation} in the Banach space $B$:
\begin{equation} \label{main}
\psi(x) = (I + R(x)) \tilde \psi(x), \quad x \in [0, \pi].
\end{equation}
Here $I$ is the identity operator in $B$.

Now suppose that the problem $L$ and the data $\tilde G = \{ \tilde \la_n, \tilde M_n \}_{n = 0}^{\iy}$ satisfy the conditions of Theorem~\ref{thm:loc}. We choose $\de_0$ to be so small that the values $\{ \tilde \la_n \}_{n = 0}^N$ definitely lie inside $\ga_N$ and the values $\{ \tilde \la_n \}_{n > N}$ definitely lie outside $\ga_N$. It is not known whether the data $\tilde G$ correspond to any problem $\tilde L$ or not.
Let $\psi(x)$ and $R(x)$ be constructed by $L$ and $\tilde G$ via the formulas above.
Then the following assertion holds.

\begin{lem}
For each fixed $x \in [0, \pi]$, the following estimate is valid: 
\begin{equation} \label{estR}
\| R(x) \|_{B \to B} \le C \de,
\end{equation}
where the constant $C$ does not depend on $x$, $\de$ and on the choice of $\tilde G$, satisfying the conditions of Theorem~\ref{thm:loc}.
\end{lem}

\begin{proof}
In order to prove~\eqref{estR}, it is sufficient to obtain the similar estimates for $\| R_{CC}(x) \|_{B_C \to B_C}$, $\| R_{CD}(x) \|_{B_D \to B_C}$, $\| R_{DC}(x) \|_{B_C \to B_D}$ and $\| R_{DD}(x) \|_{B_D \to B_D}$.
Using~\eqref{estM} and~\eqref{RCC}, we get
$$
\| R_{CC}(x) \|_{B_C \to B_C} \le \frac{1}{2\pi}\mbox{length}(\ga_N) \cdot \max_{\la, \xi \in \ga_N} |D(x, \la, \xi)| \cdot \max_{\xi \in \ga_N} |\hat {\mathcal M}_N(\xi)| \le C \de.
$$

The standard estimates (see \cite[Lemma~1.6.2]{FY01}) imply
$$
|D(x, \la, \tilde \la_n) - D(x, \la, \la_n)| \le \frac{C \exp(|\mbox{Im}\,\rho|x) \xi_n}{|\rho - n| + 1}, \quad n \ge 0.
$$
Combining the latter relation with~\eqref{RCD}, \eqref{estxi} and the obvious estimates 
\begin{equation} \label{estMn}
|M_n| \le C, \quad |\tilde M_n - M_n| \le \xi_n, \quad n \ge 0, 
\end{equation}
we get
\begin{align*}
    \| R_{CD}(x) \|_{B_D \to B_C} & \le \max_{\la \in \ga_N} \sum_{k = 1}^{\iy} (|\tilde M_{N + k}- M_{N + k}| |D(x, \la, \tilde \la_{N + k})| \\ & + |M_{N + k}| |D(x, \la, \tilde \la_{N + k}) - D(x, \la, \la_{N + k})| + \xi_{N + k} |M_{N + k}| |D(x, \la, \la_{N + k}|) \\ & \le C \sum_{n = N + 1}^{\iy} \frac{\xi_n}{|\rho - n| + 1} \le C \left(  \sum_{n = N + 1}^{\iy} (n \xi_n)^2 \right)^{1/2} \le C \de.
\end{align*}
One can similarly study the components $R_{DC}(x)$ and $R_{DD}(x)$ and finally arrive at the assertion of the lemma.
\end{proof}

\begin{cor} \label{cor:uniq}
There exists $\de_0 > 0$ such that, for every $\de \le \de_0$ and $x \in [0, \pi]$, the estimate $\| R(x) \|_{B \to B} \le \frac{1}{2}$ holds. In this case, for each fixed $x \in [0, \pi]$, the operator $(I + R(x))$ has a bounded inverse, and the main equation~\eqref{main} has a unique solution.
\end{cor}

\section{Proofs}

The aim of this section is to prove Theorems~\ref{thm:loc} and~\ref{thm:disc}.
Using the solution of the main equation~\eqref{main}, we construct the values $\tilde q(x)$, $\tilde h$ and $\tilde H$, being the solution of Inverse Problem~\ref{ip:M} for $\tilde G$.
Furthermore, stability estimates~\eqref{stab} are proved. We just outline the general strategy, since the detailed proofs are analogous to \cite[Section~1.6.2]{FY01}.

Let $L$ be a fixed BVP, and let $\de_0$ satisfy the conditions of Corollary~\ref{cor:uniq}. Let $\tilde G = \{ \tilde \la_n, \tilde M_n \}_{n = 0}^{\iy}$ be arbitrary fixed data, satisfying the conditions of Theorem~\ref{thm:loc}. Then the main equation~\eqref{main} constructed by $L$ and $\tilde G$ has a unique solution for each fixed $x \in [0, \pi]$. Denote this solution by $\tilde \psi(x) = (\tilde \psi_C(x), \tilde \psi_D(x))$, $\tilde \psi_C(x) = \tilde \psi_C(x, \la)$, $\tilde \psi_D(x) = [\tilde \psi_n(x)]_{n = 1}^{\iy}$.

\begin{lem} \label{lem:psi}
The functions $\tilde \psi_C(x, \la)$ for each fixed $\la \in \ga_N$ and $\tilde \psi_n(x)$ for $n \ge 1$ are continuously differentiable with respect to $x \in [0, \pi]$. Moreover, the following estimates hold for $x \in [0, \pi]$, $\nu = 0, 1$:
\begin{gather*}
|\tilde \psi^{(\nu)}(x, \la)| \le C, \quad 
|\tilde \psi^{(\nu)}_C(x, \la) - \psi^{(\nu)}_C(x, \la)| \le C \de, \quad \la \in \ga_N, \\
|\tilde \psi^{(\nu)}_n(x)| \le C n^{\nu}, \quad
|\tilde \psi_n(x) - \psi_n(x)| \le C \de \eta_n, \quad
|\tilde \psi'_n(x) - \psi'_n(x)| \le C \de, \quad n \ge 1, \\
\eta_n := \left( \sum_{k = 1}^{\iy} \frac{1}{k^2 (|n - k| + 1)^2}\right)^{1/2},
\end{gather*}
where the constant $C$ depends only on $L$ and $\de_0$.
\end{lem}

Define the function
\begin{align*}
\tilde \vv(x, \la) & := \vv(x, \la) - \frac{1}{2 \pi i} \int_{\ga_N} D(x, \la, \xi) \hat {\mathcal M}_N(\xi) \tilde \psi_C(x, \xi) \, d\xi \\ & - \sum_{k = 1}^{\iy} (\tilde M_{N + k} D(x, \la, \tilde \la_{N + k}) \tilde \psi_{2k-1}(x) - M_{N + k} D(x, \la, \la_{N + k}) (\tilde \psi_{2k-1}(x) + \xi_{N + k} \tilde \psi_{2k}(x))).
\end{align*}
It is easy to check that
\begin{gather*}
\tilde \vv(x, \la) = \tilde \psi_C(x, \la), \quad \la \in \ga_N, \\
\tilde \vv(x, \tilde \la_{N + k}) = \tilde \psi_{2k-1}(x), \quad
\tilde \vv(x, \la_{N + k}) = \tilde \psi_{2k-1}(x) + \xi_{N + k} \tilde \psi_{2k}(x), \quad k \ge 1.
\end{gather*}
Consequently, Lemma~\ref{lem:psi} implies
\begin{equation} \label{estphi}
\def\arraystretch{1.5}
\left.
\begin{array}{l}
    |\tilde \vv^{(\nu)}(x, \la)| \le C, \quad |\tilde \vv^{(\nu)}(x, \la) - \vv^{(\nu)}(x, \la)| \le C \de, \quad \la \in \ga_N, \\
    |\tilde \vv^{(\nu)}(x, \la_{N + k})| \le C k^{\nu}, \quad
    |\tilde \vv(x, \la_{N + k}) - \vv(x, \la_{N + k})| \le C \de \eta_k, \\
    |\tilde \vv'(x, \la_{N + k}) - \vv'(x, \la_{N + k})| \le C \de, \quad k \ge 1, \quad
    \nu = 0, 1, \quad x \in [0, \pi].
\end{array} \right\}   
\end{equation}
The similar estimates also hold for $\la_{N + k}$ replaced by $\tilde \la_{N + k}$.

Introduce the functions
\begin{align} \nonumber
    \eps_0(x) & := \frac{1}{2 \pi i} \int_{\ga_N} \hat {\mathcal M}_N(\xi) \vv(x, \xi) \tilde \vv(x, \xi) \, d\xi + \sum_{n = N + 1}^{\iy} (\tilde M_n \vv(x, \tilde \la_n) \tilde \vv(x, \tilde \la_n) \\ \label{defeps} & - M_n \vv(x, \la_n) \tilde \vv(x, \la_n)), \quad
    \eps(x) := -2 \eps_0'(x).
\end{align}

Using~\eqref{estM}, \eqref{estxi}, \eqref{estMn}, \eqref{estphi} and \eqref{defeps}, we prove the following lemma.

\begin{lem} \label{lem:eps}
The integral $\int_{\ga_N}(\dots) \, d\xi$ and the series $\sum\limits_{n = N + 1}^{\iy}(\dots)$ in \eqref{defeps} converge absolutely and uniformly with respect to $x \in [0, \pi]$. The function $\eps_0$ is absolutely continuous on $[0, \pi]$, and $\eps \in L_2(0, \pi)$. In addition, 
$$
    \max_{x \in [0, \pi]} |\eps_0(x)| \le C \de, \quad \| \eps \|_{L_2(0, \pi)} \le C \de.
$$
\end{lem}

Define 
$$
\tilde q(x) = q(x) + \eps(x), \quad \tilde h = h - \eps_0(0), \quad \tilde H = H + \eps_0(\pi).
$$

Lemma~\ref{lem:eps} implies that $\tilde q \in L_2(0, \pi)$, and the estimates~\eqref{stab} hold. Consider the BVP $\tilde L = L(\tilde q(x), \tilde h, \tilde H)$ and the function
$$
\tilde M(\la) := \sum_{n \in \tilde S} \sum_{\nu = 0}^{\tilde m_n - 1} \frac{\tilde M_{n + \nu}}{(\la - \tilde \la_n)^{\nu + 1}}.
$$

\begin{lem} \label{lem:Weyl}
The function $\tilde M(\la)$ is the Weyl function of $\tilde L$.
\end{lem}

Thus, Corollary~\ref{cor:uniq} and Lemmas~\ref{lem:psi}-\ref{lem:Weyl} together prove Theorem~\ref{thm:loc}.

\begin{proof}[Proof of Theorem~\ref{thm:disc}]
We prove Theorem~\ref{thm:disc} by reduction to Theorem~\ref{thm:loc}.
For simplicity, suppose that the problem $L$ has the only multiple eigenvalue $\la_0$ of multiplicity $m := m_0$. The general case is completely similar, but more technically complicated. In our special case, we have
\begin{equation} \label{sm1}
\hat {\mathcal M}_N(\la) = \sum_{j = 0}^{m - 1} \frac{\tilde M_j}{\la - \tilde \la_j} - \sum_{j = 0}^{m-1} \frac{M_j}{(\la - \la_0)^{j + 1}}.
\end{equation}
Using the obvious relation
$$
\frac{1}{\la - \tilde \la_j} = \frac{1}{\la - \la_0} \left( 1 + \frac{\tilde \la_j - \la_0}{\la - \tilde \la_j}\right),
$$
we obtain
\begin{equation} \label{sm2}
\frac{1}{\la - \tilde \la_j} = \sum_{s = 0}^{2 (m - 1)} \frac{(\tilde \la_j - \la_0)^s }{(\la - \la_0)^{s + 1}} + \frac{(\tilde \la_j - \la_0)^{2 m - 1}}{(\la - \la_0)^{2 m - 1}(\la - \tilde \la_j)}.
\end{equation}
Substituting~\eqref{sm2} into~\eqref{sm1}, we derive
\begin{align*}
\hat {\mathcal M}_N(\la) & = \sum_{s = 0}^{m - 1} \frac{1}{(\la - \la_0)^{s + 1}} \left( \sum\limits_{j = 0}^{m - 1} (\tilde \la_j - \la_0)^s \tilde M_j - M_s \right) \\ & + \sum_{s = m}^{2 (m - 1)} \frac{1}{(\la - \la_0)^{s + 1}} \left( \sum\limits_{j = 0}^{m - 1} (\tilde \la_j - \la_0)^s \tilde M_j \right) + \sum_{j = 0}^{m-1} \frac{(\tilde \la_j - \la_0)^{2m-1} \tilde M_j}{(\la - \la_0)^{2m-1} (\la - \tilde \la_j)}.
\end{align*}

For an appropriate choice of $\de_0$, we have $|\la - \tilde \la_j| \ge c_0$, $j = \overline{0, m-1}$, $|\la - \la_0| \ge c_0$ for $\la \in \ga_N$, where the constant $c_0 > 0$ depends only on $L$ and $\de_0$. Therefore, taking the estimates~\eqref{estMla} into account, we conclude that 
$$
\max_{\la \in \ga_N} |\hat {\mathcal M}_N(\la)| \le C \de.
$$
Using the latter relation together with the other conditions of Theorem~\ref{thm:disc}, we apply Theorem~\ref{thm:loc} and arrive at the assertion of Theorem~\ref{thm:disc}.
\end{proof}

\begin{example}
Suppose that $m_0 = 2$, $m_n = 1$ for $n \ge 2$, i.e. $N = 1$. Let us construct a special family of data $\tilde G = \{ \tilde \la_n, \tilde M_n \}$ ``close'' to $G$ in the sense of Theorem~\ref{thm:disc}. For simplicity, put $\tilde \la_n := \la_n$, $\tilde M_n := M_n$ for $n \ge 2$, i.e. only the first double eigenvalue $\la_0$ can be perturbed. In this case, the condition~\eqref{estxi} holds automatically for any $\de > 0$, and the conditions~\eqref{estMla} take the form
\begin{gather} \label{ex1}
    |\tilde M_0 + \tilde M_1 - M_0| \le \de, \\ \label{ex2}
    |\tilde M_0(\tilde \la_0 - \la_0) + \tilde M_1(\tilde \la_1 - \la_0) - M_1| \le \de, \\ \nonumber
    |\tilde M_0(\tilde \la_0 - \la_0)^2 + \tilde M_1 (\tilde \la_1 - \la_0)^2| \le \de, \\ \nonumber
    |\tilde \la_j - \la_0| \le \sqrt{\de}, \quad |\tilde M_j| \le \frac{1}{\sqrt{\de}}, \quad j = 0, 1.
\end{gather}

Fix $\de > 0$ and put
$$
\tilde M_0 := \frac{a}{\sqrt{\de}} + M_0, \quad \tilde M_1 := -\frac{a}{\sqrt{\de}}, \quad
\tilde \la_0 := \la_0 + \sqrt{\de}, \quad \tilde \la_1 := \la_0 - \sqrt{\de} + c \de, \quad
a := \frac{M_1}{2}, \quad c := \frac{M_0}{a}.
$$
One can easily check that
$$
\tilde M_0 (\tilde \la_0 - \la_0)^{\nu} + \tilde M_1 (\tilde \la_1 - \la_0)^{\nu} = M_{\nu}, \quad \nu = 0, 1, 
$$
which implies~\eqref{ex1} and~\eqref{ex2}. It can be also checked that
\begin{equation} \label{ex3}
|\tilde M_0(\tilde \la_0 - \la_0)^2 + \tilde M_1 (\tilde \la_1 - \la_0)^2| \le C \de, \quad
|\tilde \la_j - \la_0| \le C \sqrt{\de}, \quad |\tilde M_j| \le \frac{C}{\sqrt{\de}}, \quad j = 0, 1,
\end{equation}
for $\de \in (0, \de_0]$, where $\de_0$ is sufficiently small, and the constant $C$ depends only on $L$ and $\de_0$. Despite the constant $C$ in \eqref{ex3}, we can apply Theorem~\ref{thm:disc} and conclude that, for sufficiently small $\de_0 > 0$ and $\de \in (0, \de_0]$, there exists the solution $(\tilde q, \tilde h, \tilde H)$ of Inverse Problem~\ref{ip:M} for the data $\tilde G$, and the estimates~\eqref{stab} hold.

An interesting feature of this example is that the eigenvalues $\tilde \la_0$, $\tilde \la_1$ are $\sqrt{\de}$-close to $\la_0$ and the generalized weight numbers $\tilde M_0$, $\tilde M_1$ are sufficiently large for sufficiently small $\de$. Nevertheless, the potential $\tilde q$ is $C \de$-close to $q$ in $L_2$-norm. Analogous examples can be constructed for $m_0 > 2$ and eigenvalues $\{ \tilde \la_n \}_{n = 0}^{m_0-1}$, being $\de^{1/m_0}$-close to $\la_0$.
\end{example}

\section{Case of Dirichlet Boundary Conditions}

In this section, we formulate the results similar to Theorems~\ref{thm:loc} and~\ref{thm:disc} for the case of the Dirichlet boundary conditions. Since the proofs for different types of boundary conditions are quite similar, we provide only formulations in this section.  

Consider the boundary problem $L_0 = L_0(q(x))$ for equation~\eqref{eqv} with the complex-valued potential $q \in L_2(0, \pi)$ and the Dirichlet boundary conditions
\begin{equation} \label{dbc}
    y(0) = y(\pi) = 0.
\end{equation}
Denote by $\{ \la_n \}_{n = 1}^{\iy}$ the eigenvalues of $L$ counted with their multiplicities and numbered so that $|\la_n| \le |\la_{n + 1}|$, $n \ge 1$. Equal eigenvalues are consequtive.

The eigenvalue multiplicities $\{ m_n \}_{n = 1}^{\iy}$ for the problem $L_0$ are introduced similarly to the case of the Robin boundary conditions~\eqref{bc}. Note that the asymptotic formula \eqref{asymptla} is valid for $n \ge 1$, so we choose the minimal $N = N(L_0)$ such that $m_n = 1$ for $n > N$ and $|\la_N| < |\la_{N + 1}|$. Fix the contour $\ga_N = \{ \la \in \mathbb C \colon |\la| = r\}$, $r \in (|\la_N|, |\la_{N + 1}|)$.

Let $\Phi(x, \la)$ be the solution of equation~\eqref{eqv} under the conditions $\Phi(0, \la) = 1$, $\Phi(\pi, \la) = 0$. The Weyl functions is defined as $M(\la) := \Phi'(0, \la)$. Define
\begin{gather*}
M_{n + \nu} := \Res_{\la = \la_n} (\la - \la_n)^{\nu} M(\la), \quad n \in S, \quad \nu = \overline{0, m_n-1}, \\
S := \{ n \in \mathbb N \colon n = 1 \, \text{or} \, \la_n \ne \la_{n-1} \}.
\end{gather*}

The following inverse problem is analogous to Inverse Problem~\ref{ip:M}.

\begin{ip} \label{ip:MD}
Given the data $G := \{ \la_n, M_n \}_{n = 1}^{\iy}$ for the problem $L_0(q(x))$, find the potential $q$.
\end{ip}

Consider the data $\tilde G := \{  \tilde \la_n, \tilde M_n \}_{n = 1}^{\iy}$. Define the function $\hat {\mathcal M}_N(\la)$ via the formulas~\eqref{defMN}, using the data $\{ \la_n, M_n \}_{n = 1}^N$ and $\{ \tilde \la_n, \tilde M_n \}_{n = 1}^N$ defined in this section and
$$
S_N = S \cap \{ 1, \dots, N \}, \quad \tilde S_N = S \cap \{ 1, \dots, N \}.
$$
Put $\rho_n := \sqrt{\la_n}$, $\arg \rho_n \in \left[ -\tfrac{\pi}{2}, \tfrac{\pi}{2}\right)$, and $\xi_n := |\rho_n - \tilde \rho_n| + n^{-2} |M_n - \tilde M_n|$ for $n \ge 1$.

\begin{thm} \label{thm:locD}
Let $L_0 = L_0(q(x))$ be a fixed BVP in the form~\eqref{eqv},\eqref{dbc}, $N = N(L_0)$, $\ga_N = \ga_N(L_0)$. Then there exists $\de_0 > 0$ (depending on $L_0$) such that for any $\de \in (0, \de_0]$ and any complex numbers $\tilde G = \{ \tilde \la_n, \tilde M_n \}_{n = 1}^{\iy}$ satisfying the conditions \eqref{estM} and~\eqref{estxi} there exists a complex-valued function $\tilde q \in L_2(0, \pi)$ being the solution of Inverse Problem~\ref{ip:MD} for $\tilde G$. Moreover,
\begin{equation} \label{estq}
\| q - \tilde q \|_{L_2(0, \pi)} \le C \de,
\end{equation}
where the constant $C > 0$ depends only on $L_0$ and $\de_0$.
\end{thm}

\begin{thm}
Let $L_0 = L_0(q(x))$ be a fixed BVP in the form~\eqref{eqv},\eqref{dbc}, $N = N(L_0)$. Then there exists $\de_0 > 0$ (depending on $L_0$) such that for any $\de \in (0, \de_0]$ and any complex numbers $\tilde G = \{ \tilde \la_n, \tilde M_n \}_{n = 1}^{\iy}$ satisfying the conditions \eqref{estxi}, \eqref{dist} for $n, k \ge 1$ and~\eqref{estMla} there exists a complex-valued function $\tilde q \in L_2(0, \pi)$ being the solution of Inverse Problem~\ref{ip:MD} for $\tilde G$. Moreover, the estimate~\eqref{estq} holds.
\end{thm}

\begin{remark}
The arguments of Sections 3 and 4 imply that Theorems~\ref{thm:loc} and~\ref{thm:locD} are also valid, if we change $\hat {\mathcal M}_N(\la)$ to $\hat M(\la)$ in~\eqref{estM}.
\end{remark}

\medskip

{\bf Acknowledgments.} This work was supported by Grant 20-31-70005 of the Russian Foundation for Basic Research.

\noindent Natalia Pavlovna Bondarenko \\
1. Department of Applied Mathematics and Physics, Samara National Research University, \\
Moskovskoye Shosse 34, Samara 443086, Russia, \\
2. Department of Mechanics and Mathematics, Saratov State University, \\
Astrakhanskaya 83, Saratov 410012, Russia, \\
e-mail: {\it BondarenkoNP@info.sgu.ru}

\end{document}